\RequirePackage{fix-cm}
\documentclass[smallextended]{svjour3forpreprint}       
\smartqed  
\usepackage{graphicx}
\usepackage{amsmath,bm}
\usepackage{amssymb}
\usepackage[abbrev]{amsrefs}
\usepackage{mathtools}
\usepackage{enumerate}
\spnewtheorem*{notation}{Notation}{\bf}{\rm}
\newcommand{\bone}{\mathbf{1}}
\newcommand{\sg}{\sigma}
\newcommand{\N}{\mathbb{N}}
\newcommand{\R}{\mathbb{R}}
\newcommand{\cN}{\mathcal{N}}
\newcommand{\cW}{\mathcal{W}}
\newcommand{\n}{{\bf n}}
\newcommand{\dl}{\delta}
\newcommand{\wg}{\wedge}
\newcommand{\eps}{\varepsilon}
\newcommand{\zt}{\zeta}
\newcommand{\om}{\omega}
\newcommand{\cl}[1]{\mkern 1.5mu\overline{\mkern-1.5mu#1\mkern-1.5mu}\mkern 1.5mu}
\DeclareMathOperator*{\esssup}{ess\,sup}
\newcommand{\relmiddle}[1]{\mathrel{}\middle#1\mathrel{}}

\numberwithin{equation}{section}
\allowdisplaybreaks[3]
\begin{document}

\title{Pathwise uniqueness and non-explosion property of Skorohod SDEs with a class of non-Lipschitz coefficients and non-smooth domains\thanks{This study was supported by JSPS KAKENHI Grant Number JP19H00643.}
}

\titlerunning{Pathwise uniqueness and non-explosion property of Skorohod SDEs}        

\author{Masanori Hino         \and
        Kouhei Matsuura  \and Misaki Yonezawa
}


\institute{M. Hino \at
              Department of Mathematics, Kyoto University, Kyoto 606-8502, Japan \\
              \email{hino@math.kyoto-u.ac.jp}           
           \and
            K. Matsuura \at
              Institute of Mathematics, University of Tsukuba, 1-1-1, Tennodai, Tsukuba, Ibaraki, 305-8571, Japan\\
              \email{kmatsuura@math.tsukuba.ac.jp}
           \and
           M. Yonezawa \at
              Daiwa Securities Co. Ltd., Chiyoda-ku, Tokyo 100-6752, Japan\\
              \email{misaki.yonezawa@daiwa.co.jp}
}

\date{Received: date / Accepted: date}

\maketitle

\begin{abstract}
Here we study stochastic differential equations with a reflecting boundary condition. 
We provide sufficient conditions for pathwise uniqueness and non-explosion property of solutions in a framework admitting non-Lipschitz continuous coefficients and non-smooth domains.
\end{abstract}

\section{Introduction}

Let $w=\{w(t)\}_{t\ge0}$ be a one-dimensional Brownian motion on $\R$ starting in $[0,\infty)$. A reflecting Brownian motion $\xi=\{\xi(t)\}_{t\ge0}$ on $[0,\infty)$ is characterized by the solution of the following (pathwise) equation:
\begin{align}\label{eq:sk}
\begin{cases}
\xi=w+\phi, \\
\phi \text{ is non-decreasing on $[0,\infty)$, $\phi(0)=0$, and} \\
\phi(t)=\int_{0}^{t}\bone_{\{0\}}(\xi(s))\,d\phi(s),\quad t \ge 0.
\end{cases}
\end{align}
Equation~\eqref{eq:sk} for a continuous function $w=\{w(t)\}_{t\ge0}$ with a nonnegative initial value is called the Skorohod problem for $((0,\infty),w)$. This equation has a unique solution described as 
\begin{align*}
\xi(t)=
\begin{cases}
w(t),& 0 \le t \le \tau, \\
w(t)-\inf \{ w(s) \mid \tau \le s \le t \},& t>\tau,
\end{cases}
\end{align*}
where $\tau=\inf\{s>0 \mid w(s)<0\}$. Given a multidimensional domain $D \subset \R^d$ and an $\R^d$-valued continuous function $w$ on $[0,\infty)$, the Skorohod problem for $(D,w)$ can be considered similarly to \eqref{eq:sk} (see \cite{S} for a precise formulation).
Tanaka~\cite[Theorem~2.1]{T} showed that the Skorohod problem has a unique solution if $D$ is a convex domain. Saisho~\cite[Theorem~4.1]{S} extended this result to more general domains satisfying conditions {\bf (A)} and {\bf (B)}, which are defined in Section~2. The class of domains $D$ satisfying these conditions includes all convex domains and domains with a bounded $C^2$-boundary, and admits some non-smoothness.

The Skorohod problem is generalized to a stochastic differential equation (SDE) as follows. Let $(\Omega ,\mathcal{F},P)$ be a probability space. Let $D$ be a domain of $\R^d$ and denote its closure by $\cl{D}$. Given an $\R^d$-valued function $b$ and a $d \times d$ matrix-valued function $\sg$ on $[0,\infty) \times \Omega \times \cl{D}$, we are concerned with the following SDE:
\begin{eqnarray}\label{eq:sk2}
\begin{cases}
dX(t)=\sg(t,\cdot,X(t))\,dB(t)+b(t,\cdot,X(t))\,dt+d\Phi_{X}(t), &  t \ge 0, \\
X(0) \in \cl{D}. 
\end{cases}
\end{eqnarray}
Here, $\{B(t)\}_{t \ge 0}$ denotes a $d$-dimensional Brownian motion, and $\Phi_X$ is a reflection term, which is an unknown continuous function of bounded variation with properties \eqref{eq:lt1}, \eqref{eq:lt2}, and \eqref{eq:lt0}, which are presented below. Equation~\eqref{eq:sk2} is called a Skorohod SDE, which is a natural generalization of the Skorohod problem. If $b=0$ and $\sg$ is the identity matrix, the solution to \eqref{eq:sk2} is simply a reflecting Brownian motion on $\cl{D}$. If $D$ satisfies conditions {\bf (A)} and {\bf (B)}, and if coefficients $\sg$ and $b$ depend on only $x \in \cl{D}$ and are bounded continuous functions, then a solution exists for \eqref{eq:sk2} (see \cite[Remark~5.1]{S} and also Remark~\ref{rem:def} below).  If $\sg$ and $b$ depend on only $x \in \cl{D}$ and are Lipschitz continuous on $ \cl{D}$, a standard argument by the Gronwall inequality leads us to pathwise uniqueness of \eqref{eq:sk2} (\cite[Lemma~5.6]{S}). 
We are therefore interested in the case where $\sg$ and $b$ are not necessarily Lipschitz continuous.
Since \cite{Z} obtained a satisfactory sufficient condition for pathwise uniqueness in the one-dimensional case, we consider general dimensions.

Note that for usual SDEs without reflection terms, pioneering works by Yamada and Watanabe~\cites{YW1,YW2} have already treated non-Lipschitz coefficients. 
Although a number of related studies have been done since then, here we cite only a couple of works directly related to this paper \cites{FZ,L}. 
The arguments in \cite{FZ} were adapted in \cite{BY} for study of the Skorohod SDE~\eqref{eq:sk2} and to obtain sufficient conditions for pathwise uniqueness when $D$ is admissible, that is, when $D$ is roughly described as
\[
D=\{x \in \R^d \mid \varphi(x)>0\}\quad\text{and}\quad \partial D=\{x \in \R^d \mid \varphi(x)=0 \}
\]
for some nice $C^2$-function $\varphi$. In particular, $D$ has a smooth boundary.
Their arguments depend on this smoothness. 

In this paper, we consider the Skorohod SDE \eqref{eq:sk2} for domains $D$ satisfying only condition {\bf (A)} (and {\bf (B)} for some claims). First, we prove pathwise uniqueness under certain conditions that allow non-Lipschitz coefficients $\sg$ and $b$ (Theorem~\ref{thm:1}).
Our arguments are based on those of \cite[Theorem~2]{L} and inherit the assumptions in that work.
Note that condition~{\bf (A)} does not necessarily hold even when the boundary of $D$ is in $C^1$.
However, this is consistent with the fact that the pathwise uniqueness of a reflecting Brownian motion on $\cl{D}$ can fail for $C^1$ domains. See \cite[Theorem~4.1]{BB} for such an example. 

Since the solution to \eqref{eq:sk2} may explode in general, we provide two sufficient conditions for the solution to be not explosive (Theorems~\ref{thm:3} and \ref{thm:2}). 
In Theorem~\ref{thm:3}, the main conditions are described using a Lyapunov-type function.
This result corresponds to \cite[Theorem~1]{L}, where the non-explosion property of SDEs without reflection terms was discussed.
As seen in Example~\ref{ex:exc}, these conditions fit for convex domains and domains whose boundaries are (globally) described by smooth functions. In both cases, the norms of coefficients $\sg$ and $b$ in \eqref{eq:sk2} can grow as $|x|(\log|x|)^{1/2}$ in $x$ as $|x| \to \infty$. 
In Theorem~\ref{thm:2}, we discuss the non-explosion property for more general domains in a sense, but with more restrictive conditions on coefficients. 
The following are typical examples where Theorem~\ref{thm:2} can be applied:
\begin{itemize}
\item The coefficients grow by at most $|x|(\log|x|)^{1/2}$ in $x$ as $|x| \to \infty$ and are bounded near the boundary (Example~\ref{ex:exc2}(1));
\item The coefficients have at most sub-linear growth of order $(1/2)-\eps$ in $x$ for some $\eps>0$ as $|x| \to \infty$ (Example~\ref{ex:exc2}(2)).
\end{itemize}
The proof of Theorem~\ref{thm:2} is quite different from that of Theorem~\ref{thm:3}; the key idea is to prove that solutions of \eqref{eq:sk2} approach the boundary of $D$ only a finite number of times in a sense almost surely\ on each finite time interval (Lemma~\ref{lem:cons3}). 

The remainder of this paper is organized as follows: In Section~2, we set up our framework, state its main theorems (Theorems~\ref{thm:1}, \ref{thm:3}, and \ref{thm:2}), and present some typical examples. 
In Section~3, we use the arguments in \cite{L} to prove Theorems~\ref{thm:1} and \ref{thm:3}. As an application of Theorem~\ref{thm:1}, we give a sufficient condition for \eqref{eq:sk2} to have a strong solution (Corollary~\ref{cor:1}).  
In Section~4, we prove Theorem~\ref{thm:2}.
\begin{notation}
The following symbols are used in the paper.
\begin{itemize}
\item $\R_{+}$ is the half-line $[0,\infty)$.
\item $\R^d \otimes \R^d$ denotes the set of all real square matrices of size $d$, and   $A^{\ast}$ denotes the transpose of $A \in \R^d \otimes \R^d$.
\item $\langle\cdot,\cdot\rangle$ and $|\cdot|$ denote the standard inner product and norm of $\R^n$, respectively.
\item $\|\cdot\|$ denotes the Hilbert--Schmidt norm of $\R^d \otimes \R^d$.
\item For $x\in\R^d$ and $r>0$, $B(x,r)$ (resp.\ $\cl{B}(x,r)$) denotes the open (resp.\ closed) ball in $\R^d$ with center $x$ and radius $r$. 
We write $B(r)$ for $B(0,r)$.
\item For $a,b \in [-\infty,\infty]$, we write $a \vee b=\max\{a,b\}$ and $a \wg b=\min \{a,b\}$.
\item $\#I$ is the cardinality (number of elements) of a set $I$.
\item $\inf\emptyset=\infty$ by convention.
\end{itemize}
\end{notation}

\section{Main Results}
Let $D$ be a domain of $\R^d$. Let $\cl{D}_{\Delta}$ denote $\cl{D}$ or $\cl{D} \cup \{\Delta\}$ (the one-point compactification of $\cl{D}$) according whether $\cl{D}$ is compact or non-compact, respectively.  
We define
\[
\cW(\cl{D}_{\Delta})=\left\{w \colon \R_{+} \to \cl{D}_{\Delta}\relmiddle|\parbox{0.57\textwidth}{$w$ is continuous, $w(0) \in \cl{D}$,
and if $w(t)=\Delta$ for some $t \ge 0$, then $w(s)=\Delta$ for any $s \ge t$}\right\}. 
\]
For each $w \in \cW(\cl{D}_{\Delta})$, we define the lifetime $\zt(w)$ of $w$ as 
\[
\zt(w)=\inf\{t>0 \mid w(t)=\Delta \}.
\]
For $x \in \partial D(=\cl{D} \setminus D)$ and $r \in (0,\infty)$, we define $\cN_{x,r}$ and $\cN_{x}$ as
\[
\cN_{x,r}= \{\n \in \R^d \mid |\n|=1,\ B(x-r\n,r)  \cap D =\emptyset \}, \quad
\cN_{x}=\bigcup_{r \in (0,\infty)}\cN_{x,r}.
\]
An element in $\cN_x$ is called an inward normal unit vector at $x$. 

Let $(\Omega, \mathcal{F},\{ \mathcal{F}_t\}_{t \ge 0}, P)$ be a filtered probability space with the usual condition. The expectation with respect to $P$ is denoted by $E[\cdot]$. We denote by $E[\cdot:A]$ the expectation on event $A \in \mathcal{F}$. Let $\sg\colon\R_{+} \times \Omega \times \cl{D} \to \R^d \otimes \R^d$ and $b\colon \R_{+} \times \Omega  \times \cl{D} \to \R^d$ be measurable functions. Throughout this paper, we assume the following:
\begin{itemize}
\item  $\R_{+} \times \Omega \ni (t,\om) \mapsto \sg(t,\om,x)$ and $\R_{+} \times \Omega \ni (t,\om ) \mapsto b(t,\om,x)$ are progressively measurable for any fixed $x \in \cl{D}$.
\end{itemize}

We now provide an explicit definition of the Skorohod SDE.
\begin{definition}
A pair of $\{ \mathcal{F}_t\}_{t \ge 0}$-adapted processes $(X,\Phi_{X})$ is called a solution of the Skorohod SDE 
\begin{eqnarray}\label{eq:sk3}
\begin{cases}
dX(t)=\sg(t,\cdot,X(t))\,dB(t)+b(t,\cdot,X(t))\,dt+d\Phi_{X}(t), \quad t \ge 0 ,& \\
X(0) \in \cl{D} &
\end{cases}
\end{eqnarray}
if the following conditions are satisfied:
\begin{itemize}
\item  For $P$-a.s., $X=\{X(t)\}_{t \ge0 }$ belongs to $\cW(\cl{D}_{\Delta})$.
\item  For $P$-a.s., $\Phi_X=\{\Phi_X(t)\}_{ t \ge 0}$ is an $\R^d$-valued continuous function on $[0,\zt_X)$ of bounded variation on each compact interval, where we define $\zt_{X}(\omega)=\zt(X(\omega))$, $\omega \in \Omega$. 
Furthermore, for $P$-a.s.,
\begin{align}
\Phi_{X}(0)&=0, \label{eq:lt1}\\
|\Phi_{X}|_t&=\int_{0}^{t}\bone_{\partial D}(X(s))\,d|\Phi_{X}|_s, \quad t< \zt_X \label{eq:lt2},\\
\shortintertext{and}
\Phi_{X}(t)&=\int_{0}^{t}\n(s)\,d|\Phi_{X}|_s,\quad t<\zt_X. \label{eq:lt0}
\end{align}
Here, $|\Phi_{X}|_t$ denotes the total variation process of $\Phi_{X}$, and $\n(s) \in \cN_{X(s)}$ if $X(s) \in \partial D$.
\item For $P$-a.s.,
\begin{equation}\label{eq:m}
\int_0^t\{ \|\sg(s,\cdot,X(s))\|^2+|b(s,\cdot,X(s))|\}\,ds<\infty ,\quad t \in [0,\xi_X).
\end{equation}
\item For $P$-a.s.,
\begin{align*}
X(t)=X(0)+\int_0^t \sg(s,\cdot,X(s))\,dB(s)+\int_0^t b(s, \cdot,X(s))\,ds+\Phi_{X}(t), \\
t \in [0,\zt_X),
\end{align*}
where $\{B(t)\}_{t \ge 0}$ is a $d$-dimensional $\{\mathcal{F}_t\}_{t \ge 0}$-Brownian motion.
\end{itemize}
\end{definition}
We often say that $X$ is a solution of \eqref{eq:sk3} without referring to $\Phi_X$.

Following \cites{LS,S}, we introduce conditions {\bf (A)} and {\bf (B)} on $D$ as follows:
\begin{enumerate}[\bf (A)]
\item There exists $r_0 \in (0,\infty)$ such that for any $x \in \partial D$,
\begin{align*}
\cN_x=\cN_{x,r_0}\neq \emptyset.
\end{align*}
\item There exist $\dl \in (0,\infty)$ and $\beta \in [1,\infty)$ with the requirement that for any $x \in \partial D$ there exists a unit vector $\bone_x$ such that
\begin{align*}
\langle  \bone_x, \n \rangle  \ge 1/\beta 
\end{align*}
for any $\n \in \bigcup_{y \in B(x,\dl) \cap \partial D} \cN_y$.
\end{enumerate}
\begin{remark}[{\cite[Remark~1.1]{S}}]\label{rem:1}
Let $x \in \partial D$, $r>0$, and let $\n \in \R^d$ be a unit vector. Then, the following conditions are equivalent:
\begin{itemize}
\item $\n \in \cN_{x,r}$.
\item For any $y \in \cl{D}$,  $\langle y-x, \n \rangle + (2r)^{-1}|y-x|^2 \ge 0.$
\end{itemize}
Indeed,  $\n \in \cN_{x,r}$ if and only if $|y-(x-r\n)|\ge r$ for any $y \in \cl{D}$, that is, 
\[
|y-x|^2+2\langle y-x,r\n \rangle+r^2 \ge r^2
\quad\text{for any }y \in \cl{D}.
\]
\end{remark}

\begin{remark}\label{rem:def}
\begin{enumerate}
\item
Our definition of the Skorohod SDE is more general than that in \cite{S}, in that the lifetime is considered.
\item
By the same argument as in \cite[Theorem~5.1]{S}, equation \eqref{eq:sk3} possesses a (not necessarily strong) solution with infinite lifetime if the following are satisfied (see also \cite[Remark~5.1]{S}):
\begin{itemize}
\item $D$ satisfies conditions {\bf (A)} and {\bf (B)}.
\item For $P$-a.s., $\sg$ and $b$ are  continuous in $(t,x) \in \R_{+}\times\cl{D}$.
\item For any $T>0$,
\[\esssup_{\om \in \Omega}\sup_{t \in[0,T],\,x \in \cl{D}} \{ \|\sg(t,\om,x) \| \vee |b(t,\om,x)| \}<\infty.\]
\end{itemize}
\item In \cite[Theorem~2.2]{RS}, the authors provided sufficient conditions for equation \eqref{eq:sk3} to have a solution when $\sg$ and $b$ are only measurable.
\end{enumerate}
\end{remark}
Even if $\sigma$ and $b$ are locally bounded, we can obtain local solutions for \eqref{eq:sk3}  by Remark~\ref{rem:def}~(2). The solutions are strong once we prove the pathwise uniqueness of solutions for \eqref{eq:sk3}. We can then obtain a strong solution to \eqref{eq:sk3} from a standard localization argument. Such discussions are rigorously presented 
in Corollary~\ref{cor:1}, below.
Therefore, we first study the pathwise uniqueness of the Skorohod equation~\eqref{eq:sk3}. The definition is as follows:
\begin{definition}
We say that pathwise uniqueness of solutions for \eqref{eq:sk3} holds if for any two solutions $(X,\Phi_X)$ and $(Y,\Phi_{Y})$ of \eqref{eq:sk3} that are defined on the same filtered probability space with the same $d$-dimensional Brownian motion $\{B(t)\}_{t \ge 0}$ such that $X(0)=Y(0)$ $P$-a.s., we have $X(t)=Y(t)$ for all $t\ge 0$ $P$-a.s.
\end{definition}
\begin{remark}\label{rem:yw}
By the Yamada--Watanabe theorem~\cite[Corollary~3]{YW1}, the existence and pathwise uniqueness of solutions to \eqref{eq:sk3}
ensure the existence of a strong solution to \eqref{eq:sk3}. Note that the Yamada--Watanabe theorem was proved for SDEs without reflection terms. However, the proof is also valid for \eqref{eq:sk3}.
\end{remark}

To describe a sufficient condition for pathwise uniqueness, we introduce the following condition for a nonnegative and Borel measurable function $\Lambda$ on $[0,1)$:
\begin{itemize}
\item[{\bf (L)}] There exists some $\eps_0 \in (0,1)$ such that $\Lambda$ is continuous and non-decreasing on $[0,\eps_0)$, and 
\begin{equation}\label{eq:zt0}
\int_{0}^{\eps_0}\frac1{\Lambda(s)}\,ds=\infty.
\end{equation}
\end{itemize}
Note that \eqref{eq:zt0} implies
\begin{equation}\label{eq:zt}
\int_{0}^{\eps}\frac1{\Lambda(s)}\,ds=\infty,\quad\eps\in(0,\eps_0].
\end{equation}
For example, $\Lambda(s)=s$, $\Lambda(s)=s\log(1/s)$, and $\Lambda(s)=s\log(1/s) \times \log \log (1/s)$ satisfy the above conditions. 

Let $g=\{g(t,\cdot)\}_{t\ge0}$ be a nonnegative progressively measurable process such that, for any $T\ge 0$,
\begin{align}
\int_{0}^T g(s,\cdot)\,ds <\infty\quad P\text{-a.s.} \label{eq:g}
\end{align}
A sufficient condition for pathwise uniqueness is given as follows:
\begin{theorem}\label{thm:1}
Assume condition {\bf (A)} and that for each $R>0$ there exists a Borel measurable function $\Lambda_R :[0,1) \to \R_{+}$ satisfying {\bf (L)} and for $P$-a.s.\,$\omega$,
\begin{align}
&\|\sg(t,\om,x)-\sg(t,\om,y)\|^2+2 \langle x-y ,b(t,\om,x)-b(t,\om,y)\rangle  \notag \\
&\le g(t,\omega) \Lambda_R(|x-y|^2) \label{eq:cond}
\end{align}
for any $t \ge 0$ and $x,y \in \cl{D} \cap B(R)$ with $|x-y|<1$. 
Then, the pathwise uniqueness of solutions for \eqref{eq:sk3} holds. 
\end{theorem}
Combining Theorem~\ref{thm:1}, and Remarks~\ref{rem:def}(2) and \ref{rem:yw}, we obtain the following sufficient condition for \eqref{eq:sk3} to have a strong solution:
\begin{corollary}\label{cor:1}
Assume conditions {\bf (A)} and {\bf (B)} and the following:
\begin{enumerate}
\item  For each $R>0$ there exists a Borel measurable function $\Lambda_R :[0,1) \to \R_{+}$ satisfying {\bf (L)} and for $P$-a.s.\,$\omega$,
\begin{align*}
&\|\sg(t,\om,x)-\sg(t,\om,y)\|^2 \vee \langle x-y ,b(t,\om,x)-b(t,\om,y)\rangle \\
&\le g(t,\omega) \Lambda_R(|x-y|^2) 
\end{align*}
for any $t \ge 0$ and $x,y \in \cl{D} \cap B(R)$ with $|x-y|<1$.
\item For $P$-a.s.\ $\om \in \Omega$, the maps $\R_{+} \times \cl{D} \ni (t,x) \mapsto \sg(t,\om,x)$ and $\R_{+} \times \cl{D} \ni (t,x) \mapsto b(t,\om,x)$ are continuous.
\item For any $T>0$ and $R>0$, 
\[  \esssup_{\om \in \Omega}\sup_{t \in[0,T],\, x \in \cl{D} \cap B(R)} \{\| \sg(t,\om,x) \| \vee |b(t,\om,x)|\} <\infty.\] 
\end{enumerate}
Then, \eqref{eq:sk3} possesses a strong solution.
\end{corollary}
Next, we discuss the non-explosion property of the solution.
Let $\gamma\colon \R_{+} \to [1,\infty)$ be a continuous and non-decreasing function such that $\lim_{s \to \infty}\gamma(s)=\infty$ and 
\begin{align}\label{eq:gm}
\int_{0}^{\infty}\frac{1}{\gamma(s)}\,ds=\infty.
\end{align}
Functions
$\gamma(s)=s+1$, $\gamma(s)=s\log(s+1)+1$ are typical examples satisfying the above conditions.
\begin{theorem}\label{thm:3}
Assume condition {\bf (A)},
and that there exists a nonnegative function $V \in C^{1,2}([0,\infty) \times \R^d)$ with the following conditions:
\begin{enumerate}[\rm(V.1)]
\item  For any $t>0$,
\[
\lim_{R\to\infty}\inf_{s \in [0,t],\, x \in \cl{D} \setminus B(R)}V(s,x) =\infty.
\]
\item For any $x \in \partial D$, $t \ge 0$, and $\n \in \mathcal{N}_{x}$, $
\langle (\nabla V)(t,x),\n \rangle \le 0.$
\item For $P$-a.s.\,$\om$,
\begin{align}
&\|\sg(t,\om,x)\|^2(\Delta V)(t,x)+ 2 \langle b(t,\om,x),(\nabla V)(t,x)\rangle+  2\frac{\partial V}{\partial t}(t,x)\notag \\
& \le g(t,\omega)\gamma(V(t,x)) \label{eq:eqv}
\end{align}
for any $t \ge 0$ and $x\in \cl{D}$. 
\end{enumerate}
Then, the solutions to \eqref{eq:sk3} are non-explosive, that is, $P(\zt_{X}=\infty)=1$.
\end{theorem}
The following describes another sufficient condition for non-explosion.
For $x\in\cl{D}$, $\dl>0$, and $T>0$, we set 
\[
 M(x,\dl,T)=\esssup_{\om \in \Omega}\sup_{t \in[0,T],\,z \in B(x,\dl) \cap \cl{D}}\{ \|\sg(t,\om,z)\|^2 \vee |b(t,\om,z)|^2 \}.
\]
\begin{theorem}\label{thm:2}
Assume conditions {\bf (A)} and {\bf (B)}, and the following:
\begin{enumerate}
\item For $P$-a.s.\,$\om$,
\begin{align}
\|\sg(t,\om,x)\|^2 \vee |b(t,\om,x)|^2 \le g(t,\omega) \gamma(|x|^2) \label{eq:cond2}
\end{align}
for any $t \ge 0$ and $x\in \cl{D}$. 
\item For each $T>0$, there exist constants $C>0$, $\nu\in[0,1)$, $\hat\dl>0$, $\hat \beta\in(0,1)$, points $\{x_n\}_{n=1}^\infty\subset \partial D$, and positive numbers $\{\dl_n\}_{n=1}^\infty\subset[\hat\dl,\infty)$ such that $\partial D\subset \bigcup_{n=1}^\infty B(x_n,\hat \beta \dl_n)$ and
\begin{equation}\label{eq:thm2asmp}
 M(x_n,\dl_n,T)\le C \dl_n^\nu\quad \text{for any }n\in\N.
\end{equation}
\end{enumerate}
Then, the solutions of \eqref{eq:sk3} are non-explosive. 
\end{theorem}
\begin{example}\label{ex:exc}
The following are some examples meeting the assumptions in Theorem~\ref{thm:3}:
\begin{enumerate}
\item If $D$ is an unbounded convex domain, it satisfies conditions {\bf (A)} and {\bf (B)}. We take $x_0 \in D$ and set \[V(t,x)=|x-x_0|^2 ,\quad t \in [0,\infty),\ x \in \R^d.\]
We write $V(x)$ for $V(t,x)$ since $V(t,x)$ does not depend on $t$. $V$ satisfies (V.1). Because $D$ is convex, it follows that for any $x \in \partial D$ and $\n \in \cN_x$
\[\langle \nabla  V(x), \n \rangle =2 \langle x-x_0,\n \rangle \le 0, \] 
proving (V.2). If there exists $C>0$ such that \[\|\sg (t,\om,x)\| \vee |b(t,\om,x)| \le C[ |x| \{\log(|x|+1)\}^{1/2}+1 ]\] for any $(t,\om,x)\in [0,\infty) \times \Omega \times \cl{D}$, we then see that (V.3) holds with $\gamma(s)=s \log(s+1)+1$ $(s \ge0)$ and a sufficiently large constant function $g$.
\item Let  $H\colon [-1,\infty) \to \R$ be a smooth function such that $H(-1)=0$ and $H(s)>0$ for any $s>-1$. We also assume that there exist $m>0$ and $M>-1$ such that $H'(s) \vee H''(s)\le m$ for any $s>-1$, and 
\begin{equation}
\int_{-1}^{s}H(u)\,du \ge \frac{s^2}m \label{eq:eqh}
\end{equation}
for any $s>M$. We define a domain $D \subset \R^d$ as
\[
D=\{(x_1,x_2,\ldots,x_d) \in\R^d \mid x_1>-1,\ x_2^2+\cdots+x_d^2<H(x_1)^2\}.
\]
Since $H''(s)$ is bounded above, $D$ satisfies condition {\bf (A)}. Moreover, $\cN_x$ is a singleton for any $x=(x_1,x_2,\ldots,x_d) \in \partial D$ with $x_1>-1$. For each $x=(x_1,x_2,\ldots,x_d) \in \R^d$, we write $\tilde{x}=(x_2,\ldots,x_d)$ and denote the length of $\tilde{x}$ as $|\tilde{x}|$ 
 with an abuse of notation. Then, for any $x=(x_1,\tilde{x}) \in \partial D$ with $x_1>-1$, the inward unit vector $\n$ at $x$ is 
\begin{equation*}
\n=
\frac1{\sqrt{H'(x_1)^2+1}}\left(H'(x_1) ,-\frac{\tilde{x}}{|\tilde{x}|}\right).
\end{equation*}
We define a nonnegative function $V\colon [0,\infty) \times D \to \R$ as
\[
V(t,x)=\int_{-1}^{x_1}H(s)\,ds+\frac{m}{2}|\tilde{x}|^2,\quad t \in [0,\infty),\ x =(x_1,\tilde{x}) \in D.
\]
We write $V(x)$ for $V(t,x)$. Then,
$V(x)$ extends to a smooth function on $\R^d$. We see that $V$ satisfies conditions (V.1) and (V.2).
Assume that there exists $C_1>0$ such that 
\[
\|\sg(t,\om,x)\| \vee |b(t,\om,x)|\le C_1[ |x| \{\log (|x|+1)\}^{1/2}+1]
\] 
for any $(t,\om,x)\in [0,\infty) \times \Omega \times \cl{D}$. For any $x \in \cl{D}$, we then have
\begin{align*}
(\Delta V)(x)\le md, \quad
|(\nabla V)(x)|^2\le m^2(x_1+1)^2+m^2|\tilde{x}|^2.
\end{align*}
We see that the left-hand side 
 of \eqref{eq:eqv} is less than or equal to 
\[C_2 \{ |x|^2 \log (|x|+1)+1\},\quad x=(x_1,\tilde{x}) \in \cl{D}\]
for some $C_2>0$.
It follows from \eqref{eq:eqh} that there exists $C_3>0$ such that
$
V(x) \ge C_3 |x|^2
$
for any $x \in \cl{D}$. 
Therefore, (V.3) holds with $\gamma(s)=s \log(s+1)+1$ $(s \ge0)$ and a sufficiently large constant function $g$.
\end{enumerate}
\end{example}
\begin{example}\label{ex:exc2}
The following examples can apply Theorem~\ref{thm:2}.
Assume that $D$ satisfies {\bf (A)} and {\bf (B)} in both cases.
\begin{enumerate}
\item Suppose assumption~(1) in Theorem~\ref{thm:2}. Moreover, for each $T>0$, suppose that there exists $\hat\dl>0$ such that 
\[
\esssup_{\om\in\Omega}\sup_{t\in[0,T],\,x\in D(\hat\dl)}\{\|\sg(t,\om,x)\|\vee |b(t,\om,x)|\}<\infty,
\]
where $D(\hat\dl)=\bigcup_{y\in\partial D}B(y,\hat\dl)\cap \cl{D}$.
Then, assumption~(2) holds with $\nu=0$, $\hat\beta=1/2$, $\dl_n=\hat\dl$ for $n\in\N$, and $\{x_n\}_{n=1}^\infty\subset \partial D$ being taken so that $\partial D\subset\bigcup_{n\in\N}B(x_n,\hat\dl/2)$.
\item Suppose that there exist $C>0$ and $\eps \in (0,1/2)$ such that 
\[
\|\sg(t,\om,x)\| \vee |b(t,\om,x)|\le C(|x|^{1/2-\eps}+1)
\]
for any $(t,\om,x)\in [0,\infty) \times \Omega \times \cl{D}$.
Then, assumptions~(1) and (2) in Theorem~\ref{thm:2} hold with $\nu=1-2\eps$, $\hat\beta=1/2$, $\{x_n\}_{n=1}^\infty\subset \partial D$ such that $\partial D\subset\bigcup_{n\in\N}B(x_n,(|x_n|+1)/2)$, and $\dl_n=|x_n|+1$ for $n\in\N$.
\end{enumerate}
\end{example}

\section{Proofs of Theorem~\ref{thm:1}, Corollary~\protect\ref{cor:1} and Theorem~\ref{thm:3}}
We introduce the following lemma for later use:
\begin{lemma}\label{lem:vee}
\begin{enumerate}
\item Let $f$ be a nonnegative and non-decreasing function on an interval $(0,\eps]$ such that $\int_0^\eps f(t)^{-1}dt=\infty$. Then, $\int_0^\eps (f(t)\vee t)^{-1}dt=\infty$.
\item Let $h$ be a nonnegative and non-decreasing function on an interval $[r,\infty)$ such that $h \ge 1$ on $[r,\infty)$ and $\int_r^\infty h(t)^{-1}dt=\infty$. Then, $\int_r^\infty (h(t)\vee t)^{-1}dt=\infty$.
\end{enumerate}
\end{lemma}
\begin{proof}
\begin{enumerate}
\item Let $A=\{t\in(0,\eps]\mid f(t)<t\}$. If $A=\emptyset$, then the assertion is obvious. Suppose $A\ne\emptyset$ and let $a=\inf A$. If $a>0$, then
\[
 \int_0^\eps \frac1{f(t)\vee t}\,dt\ge \int_0^a \frac1{f(t)}\,dt=\infty.
\]
Otherwise, we can take a decreasing sequence $\{t_n\}_{n=1}^\infty$ converging to $0$ such that $t_n\in A$ and $t_{n+1}<t_n/2$ for every $n$. Then, because 
\[
f(t)\vee t\le f(t_n)\vee t_n=t_n,\quad t\in[t_{n+1},t_n],
\]
we obtain that
\begin{align*}
\int_0^\eps \frac1{f(t)\vee t}\,dt
&\ge\sum_{n=1}^\infty \int_{t_{n+1}}^{t_n} \frac1{f(t)\vee t}\,dt\\
&\ge\sum_{n=1}^\infty \int_{t_{n+1}}^{t_n} \frac1{t_n}\,dt\ge \sum_{n=1}^\infty \frac{t_n}2\cdot \frac1{t_n}=\infty.
\end{align*}
\item The proof is similar to that of (1).
Let $A=\{t\in[r,\infty]\mid h(t)<t\}$. If $A=\emptyset$, then the assertion is obvious. Suppose $A\ne\emptyset$ and let $a=\sup A$. If $a<\infty$, then
\[
 \int_a^\infty \frac1{h(t)\vee t}\,dt\ge \int_a^\infty \frac1{h(t)}\,dt=\infty.
\]
Otherwise, we can take an increasing sequence $\{t_n\}_{n=1}^\infty$ diverging to $\infty$ such that $t_n\in A$ and $t_{n+1}>2t_n$ for every $n$. Then, since 
\[
h(t)\vee t\le h(t_{n+1})\vee t_{n+1}=t_{n+1},\quad t\in[t_n,t_{n+1}],
\]
we obtain that
\begin{align*}
\int_r^\infty \frac1{h(t)\vee t}\,dt
&\ge\sum_{n=1}^\infty \int_{t_n}^{t_{n+1}} \frac1{h(t)\vee t}\,dt\\
&\ge\sum_{n=1}^\infty \int_{t_n}^{t_{n+1}} \frac1{t_{n+1}}\,dt\ge \sum_{n=1}^\infty \frac{t_{n+1}}2\cdot \frac1{t_{n+1}}=\infty.\tag*{\qed}
\end{align*}
\end{enumerate}
\end{proof}

\begin{proof}[of Theorem~\ref{thm:1}]
Let $(X,\Phi_X)$ and $(Y,\Phi_{Y})$ be two solutions to \eqref{eq:sk3} that are defined on the same filtered probability space with the same $d$-dimensional Brownian motion $\{B(t)\}_{t \ge 0}$ such that $X(0)=Y(0)$ $P$-a.s. We define
\begin{align*}
\eta(t)&=X(t)-Y(t),\\
\xi(t)&=|\eta(t)|^2, \\
e(t)&=\sg(t,\cdot,X(t))-\sg(t,\cdot,Y(t)), \\
f(t)&=b(t,\cdot,X(t))-b(t,\cdot,Y(t)).
\end{align*}
From the It\^o formula,
\begin{align}
d\xi(t)&=\langle 2 e^{\ast}(t)\eta(t),\,dB(t) \rangle+\langle 2 \eta(t), f(t) \rangle\,dt  \notag \\
&\quad+\langle 2 \eta(t),  d\Phi_{X}(t)-d\Phi_{Y}(t)\rangle+  \|e(t)\|^2\,dt. \label{eq:sm1}
\end{align}
Fix $R>0$.
By Lemma~\ref{lem:vee}~(1), $\Lambda_R(s)\vee s$ $(s \in [0,1))$ also satisfies the conditions imposed on $\Lambda_R$. Thus, we may assume  
\begin{align}\label{eq:Lambda}
\Lambda_R(s)\ge s\quad\text{for }s \in [0,1)
\end{align}
without loss of generality. For each $r>0$, we define $\phi_{r}\colon[0,1] \to \R_{+}$ by
\[
\phi_{r}(s)=\int_{0}^{s}\frac1{\Lambda_R(u)+r}\,du.
\]
Then, for any $s\in(0,\eps_0)$, we have
\begin{align}
\phi_r(s) &\nearrow \int_{0}^{s}\frac{1}{\Lambda_R(s)}\,du=\infty \text{ as }r \to 0 \label{eq:phi1}\\
\shortintertext{and}
\phi'_r(s)&=\frac{\partial \phi_r}{\partial s}(s)=\frac{1}{\Lambda_R(s)+r} \ge 0 \label{eq:phi2}.
\end{align}
Fix $r>0$. We take a concave function $\bar{\phi}_{r}\colon\R \to \R$ such that $\bar{\phi}_r(s)=\phi_{r}(s)$ for $s\in[0,\eps_0)$. For $M>0$, we define 
\begin{align*}
\tau_{R}&=\inf \{t>0 \mid |X(t)|\vee |Y(t)|\ge R\}, \\
\chi_{M,R}&=\tau_R\wg\inf \left\{t>0 \relmiddle| \begin{array}{l}\displaystyle\int_{0}^t \|e(s)\|^2\,ds \vee  |\Phi_X|_{t \wg \tau_R}\\
\displaystyle \vee |\Phi_Y|_{t \wg \tau_R}\vee\int_0^tg(s,\cdot)\,ds \ge M\end{array}\right\}.
\end{align*}
For each $\eps \in (0,\eps_0)$, we define $U_\eps$ by 
\begin{align*}
U_{\eps}&=\inf \{t>0 \mid |\xi(t)| \ge \eps\}.
\end{align*}
Fix $M>0$ and $\eps \in (0,\eps_0)$.  
To simplify the notation, we write
\[
\rho:= U_{\eps} \wg \chi_{M,R}.
\]
Applying the It\^o--Tanaka formula to $\bar{\phi}_r\circ\xi$ and  using \eqref{eq:sm1}, we obtain
\begin{align}
\bar{\phi}_r(\xi(\rho) )
&=\bar{\phi}_r(\xi(0))+2\int_{0}^{\rho}\bar{\phi}_r'(\xi(s))\langle e^{\ast}(s)\eta(s),\,dB(s) \rangle \notag \\
&\quad+ 2\int_{0}^{\rho }\bar{\phi}_r'(\xi(s)) \langle  \eta(s), f(s) \rangle\,ds  +\int_{0}^{\rho }\bar{\phi}_r'(\xi(s)) \| e(s)\|^2 \,ds  \notag \\
&\quad+2\int_{0}^{\rho}\bar{\phi}_r'(\xi(s)) \langle  \eta(s),  d\Phi_{X}(s)-d\Phi_{Y}(s)\rangle
+\frac12 \int_{\R} L_{\rho }^{a} \bar{\phi}_{r}''(da),  \label{eq:sm2}
\end{align}
where $\{L_t^a\}_{t \ge 0}$ denotes the local time at $a$ of the semimartingale $\{\xi(t)\}_{t \ge 0}$, and $\bar{\phi}_{r}''$ the second derivative of $\bar{\phi}_r$ in the sense of distribution. Because $\bar{\phi}_r$ is a concave function, 
the last term of \eqref{eq:sm2} is nonpositive.
For $P$-a.s., $\xi(0)=0$ and $ \bar{\phi}_r(\xi(0))=0$.  
For any $s \in [0,U_{\eps}]$, we see from \eqref{eq:cond} that
\begin{align*}
2 \langle  \eta(s), f(s) \rangle+ \| e(s)\|^2 &\le g(s,\cdot) \Lambda_R(\xi(s)). 
\end{align*}
Therefore, the sum of the third and the fourth terms of \eqref{eq:sm2} is dominated by 
\[
\int_{0}^{\rho }\frac{g(s,\cdot)\Lambda_R(\xi(s))}{\Lambda_R(\xi(s))+r}  \,ds
\le \int_{0}^{\rho }g(s,\cdot)\,ds\le M.
\]
By Remark~\ref{rem:1}, it holds that
\begin{align*}
\frac{1}{2r_0}|Y(s)-X(s)|^2 \ge \langle X(s)-Y(s),\n(s) \rangle
\end{align*}
if $X(s) \in \partial D$. 
From \eqref{eq:lt0} and \eqref{eq:lt2}, we have
\begin{align}
 \langle X(s)-Y(s),d\Phi_X(s) \rangle &= \langle X(s)-Y(s),\n(s) \rangle\,d|\Phi_{X}|_s  \notag \\
 &\le \frac{1}{2r_0}|Y(s)-X(s)|^2\, d|\Phi_X|_s. \label{eq:lt3}
\end{align}
By exchanging the roles of $X$ and $Y$, we have 
\begin{align}
 \langle Y(s)-X(s),d\Phi_Y(s) \rangle &\le \frac{1}{2r_0}|X(s)-Y(s)|^2\,d|\Phi_Y|_s. \label{eq:lt4}
\end{align}
From \eqref{eq:lt3}, \eqref{eq:lt4}, \eqref{eq:phi2}, and \eqref{eq:Lambda}, it follows that for any $t \ge 0$,
\begin{align*}
&\int_{0}^{\rho }\bar{\phi}_r'(\xi(s)) \langle  \eta(s),  d\Phi_{X}(s)-d\Phi_{Y}(s)\rangle \\
&\le  \frac{1}{2r_0}\int_{0}^{\rho}\bar{\phi}_r'(\xi(s))|X(s)-Y(s)|^2\,d(|\Phi_X|_s+|\Phi_Y|_s)  \\
&= \frac{1}{2r_0}\int_{0}^{\rho}\frac{\xi(s)}{\Lambda_{R}(\xi(s))+r}\,d(|\Phi_X|_s+|\Phi_Y|_s)  \\
&\le \frac{1}{2r_0} (|\Phi_X|_{\rho}+|\Phi_Y|_{\rho})
\le M/r_0.
\end{align*}
Combining these estimates, we get
\begin{align}
\bar{\phi}_r(\xi(\rho ))
\le 2\int_{0}^{\rho }\bar{\phi}_r'(\xi(s))\langle e^{\ast}(s)\eta(s),dB(s) \rangle
+M+\frac{2M}{r_0}.
\label{eq:sm3}
\end{align}
For any $s \in [0,\rho]$, we have 
\[
|\bar{\phi}_r'(\xi(s))e^{\ast}(s)\eta(s)|^2 \le |\bar{\phi}_r'(\xi(s))|^2 \|e(s)\|^2 |\xi(s)| \le \frac1r \|e(s)\|^2. \]
Therefore, we see that 
$\left\{\int_{0}^{t\wg\rho}\bar{\phi}_r'(\xi(s))\langle e^{\ast}(s)\eta(s),dB(s) \rangle\right\}_{t\in[0,\infty]}$
is a martingale. Then, taking the expectations of both sides of \eqref{eq:sm3}, we obtain
\[
E[\bar{\phi}_r(\xi(\rho))]\le  M   +\frac{2M}{r_0}.
\]
From the monotone convergence theorem, 
\[
E\left[  \int_{0}^{\xi(\rho)}\frac{1}{\Lambda_R(s)}\,ds \right]=\lim_{r \to 0}E[\bar{\phi}_r(\xi(\rho))]
 \le M   +\frac{2M}{r_0} <\infty.
\]
In view of \eqref{eq:zt} with $\Lambda$ replaced by $\Lambda_R$, we obtain that
$\xi(\rho)=0$ $P$-a.s.
Therefore,
\begin{align*}
0=E[\xi(\rho)]
&\ge E[\xi(U_{\eps} \wg \chi_{M,R} ):U_{\eps}<\chi_{M,R}  ] \\
&= \eps P(U_{\eps}<\chi_{M,R}),
\end{align*}
which implies that 
\[
U_{\eps}\ge \chi_{M,R} \quad\text{$P$-a.s.}
\]
By letting $\eps\downarrow0$, then $M\to\infty$, we see from \eqref{eq:m} and \eqref{eq:g} that
\[
X(t)=Y(t) \text{ for }t< \tau_{R},\quad P\text{-a.s.} 
\]
for any $R>0$.
This, in particular, implies that the lifetimes of $X$ and $Y$ are the same $P$-a.s. 
Accordingly, we have
$P(X(t)=Y(t),\,   t \ge 0)=1$.\qed
\end{proof}

\begin{proof}[of Corollary~\protect\ref{cor:1}] 
We fix $n \in \N$.
Define $1$-Lipschitz functions $u_n\colon \cl{D} \to \R$ and $v_n\colon \R_{+} \to \R$ as
\begin{align*}
u_n(x)&=0\vee (n+1-|x|)\wg 1,\quad x\in\cl{D},\\
v_n(t)&=0\vee (n+1-t)\wg 1,\quad t\in\R_+.
\end{align*}
We define functions $\sg_n$ and $b_n$ on $\R_{+} \times \Omega \times \cl{D}$ as 
\begin{align*}
\sg_n(t,\om,x)&=\sg(t,\om,x)u_n(x) v_n(t),\\
\quad b_n(t,\om,x)&=b(t,\om,x)u_n(x)v_n(t).
\end{align*}
For $P$-a.s.\ $\om \in \Omega$, the maps $\R_{+} \times \cl{D} \ni (t,x) \mapsto \sg_n(t,\om,x)$ and $\R_{+} \times \cl{D} \ni (t,x) \mapsto b_n(t,\om,x)$ are bounded continuous. 
We fix $R>0$ and set
\begin{align*}
K_{n,R}=\esssup_{\om \in \Omega }\sup_{t \in [0,n+1],\ z\in \cl{D} \cap B(R)} \{ \| \sg(t,\om,z)\| \vee |b(t,\om,z)|\}<\infty.
\end{align*}
It follows that for each $(t,\om) \in \R_{+} \times \Omega$ and $x,y \in \cl{D} \cap B(R)$,
\begin{align*}
\|\sg_n(t,\om,x)-\sg_n(t,\om,y)\| &\le u_n(x)v_n(t)\|\sg(t,\om,x)-\sg(t,\om,y)\| \\
&\quad +  \| (u_n(x)-u_n(y))v_n(t)\sg(t,\om,y)\| \\
&\le \|\sg(t,\om,x)-\sg(t,\om,y)\|  +K_{n,R} |x-y|
\end{align*}
and
\begin{align*}
\langle x-y ,b_n(t,\om,x)-b_n(t,\om,y)\rangle &=u_n(x)v_n(t)\langle x-y ,b(t,\om,x)-b(t,\om,y)\rangle \\
&\quad +\langle x-y ,v_n(t)b(t,\om,y)(u_n(x)-u_n(y))\rangle \\
&\le u_n(x)v_n(t) \langle x-y ,b(t,\om,x)-b(t,\om,y)\rangle \\
&\quad+K_{n,R} |x-y|^2.
\end{align*}
From the assumption, there exists a Borel measurable function $\Lambda_R :[0,1) \to \R_{+}$ satisfying {\bf (L)} such that  for $P$-a.s.\,$\omega$,
\begin{align*}
&\|\sg(t,\om,x)-\sg(t,\om,y)\|^2\le g(t,\omega) \Lambda_R(|x-y|^2) 
\end{align*}
and 
\begin{align*}
&\|\sg(t,\om,x)-\sg(t,\om,y)\|^2+2\langle x-y ,b(t,\om,x)-b(t,\om,y)\rangle \\
&\le g(t,\omega) \Lambda_R(|x-y|^2) 
\end{align*}
for any $t \ge 0$ and $x,y \in \cl{D} \cap B(R)$ with $|x-y|<1$. 
From these estimates, we see that for $P$-a.s.\,$\omega$ and such $t,x,y$,
\begin{align*}
&\|\sg_n(t,\om,x)-\sg_n(t,\om,y)\|^2 +
2\langle x-y ,b_n(t,\om,x)-b_n(t,\om,y)\rangle  \\
&\le   2\|\sg(t,\om,x)-\sg(t,\om,y)\|^2  +2K_{n,R}^2 |x-y|^2 \\
&\quad  + 2u_n(x)v_n(t) \langle x-y ,b(t,\om,x)-b(t,\om,y)\rangle +2K_{n,R} |x-y|^2 \\
&\le  (2-u_n(x)v_n(t))g(t,\omega) \Lambda_R(|x-y|^2) + u_n(x)v_n(t)g(t,\omega) \Lambda_R(|x-y|^2) \\
&\quad +(2K_{n,R}^2+2K_n) |x-y|^2\\
&= 2 g(t,\omega) \Lambda_R(|x-y|^2)+(2K_{n,R}^2+2K_{n,R}) |x-y|^2.
\end{align*}
By Lemma~\ref{lem:vee}(1), we may assume that $\Lambda_R(s) \ge s$ for any $s\in[0,1]$. Therefore, we see that $\sg_n$ and $b_n$ satisfy \eqref{eq:cond}. Then, by Theorem~\ref{thm:1} and Remarks~\ref{rem:yw} and~\ref{rem:def}(2), the equation 
 \begin{align}
X_n(t)&=X(0)+\int_{0}^t\sg_n(s,\cdot,X_n(s))\,dB(s)+\int_{0}^t b_n(s,\cdot,X_n(s))\,ds \notag \\
&\quad+\Phi_{X_n}(t), \quad t \ge 0 \label{eq:n}
\end{align}
has a strong solution. That is, \eqref{eq:n} has a solution for a given $d$-dimensional Brownian motion. We define $\tau_n:=\inf\{t>0 \mid X_n(t) \notin \cl{D} \cap B(n)\}.$ Then, by the definition of $\sg_n$ and $b_n$, we see that $X_n=\{X_n(t)\}_{t \ge 0}$ solves \eqref{eq:sk3} up to $n \wg \tau_n$.
For $m>n$, we have $X_m(t)=X_n(t)$, $t \in [0,n \wg \tau_n]$, by uniqueness. Then, we can define $X=\{X(t)\}_{t \ge 0}$ as $X(t)=X_n(t)$ for $t\le n \wg \tau_n$ and $n \in \N$. It is easy to see that $X$ is a solution of \eqref{eq:sk3} and $\zt_X=\lim_{n \to \infty}n \wg \tau_n$, by letting $X(t)=\Delta$ for $t\in[\zt_X,\infty)$.
  \qed
\end{proof}

\begin{proof}[of Theorem~\protect\ref{thm:3}] 
 We define $\hat{\xi}=\{\hat{\xi}(s)\}_{s \ge 0}$ by
\begin{align*}
\hat{\xi}(s)=V(s,X(s)),\quad s \ge 0.
\end{align*}
It follows from the It\^o formula that  
\begin{align}
d \hat{\xi}(s)
&= \langle  \sigma^{\ast}(s,\cdot,X(s))(\nabla V)(s, X(s)),\,dB(s) \rangle \nonumber \\
&\quad+ \langle   (\nabla V)( s,X( s)), b(s,\cdot,X(s)) \rangle\,ds +\langle (\nabla V)( s,X(s)),  d\Phi_{X}(s)\rangle  \nonumber \\
&\quad+ \frac12 \|\sigma(s,\cdot,X(s))\|^2 (\Delta V)(s,X(s))\,ds+ \frac{\partial V}{\partial s}(s,X(s))\,ds,\label{eq:eql}
\end{align}
where $\nabla$ and $\Delta$ are differentiations with respect to the second variable.
For $R>0$ and $M>0$, we set
\begin{align*}
\tau_{R}&=\inf\{s>0 \mid |X(s)| \ge R\},\\
\chi_{M}&=\inf \left\{s>0 \relmiddle| \int_0^sg(u,\cdot)\,du \ge M\right\},
\end{align*}
and define $\psi\colon\R_{+} \to \R_{+}$ as
\[
\psi(s)=\int_{0}^{s}\frac{1}{\gamma(u)}\,du,\quad s  \in \R_+.
\]
For now, we fix $t>0$, $R>0$, and $M>0$, and write
\[
\rho= t \wg \tau_{R} \wg \chi_M.
\]
Since $\gamma$ is a non-decreasing function on $\R_{+}$, $\psi$ can extend to a concave function on $\R$, which is denoted as the same symbol.
By applying the It\^o--Tanaka formula to $\psi\circ\hat\xi$ and using \eqref{eq:eql}, we have 
\begin{align}
&\psi(\hat{\xi}(\rho))=\psi(\hat \xi(0))+\int_{0}^{\rho }\psi'(\hat{\xi}(s))\,d \hat{\xi}(s)+\frac{1}{2} \int_{\R} \hat{L}_{\rho }^{a} \psi_{r}''(da) \notag  \\
&=\psi(V(0,X(0))) +\int_{0}^{\rho  }\psi'(\hat{\xi}(s))\langle  \sigma^{\ast}(s,\cdot,X(s))(\nabla V)( s,X(s)),\,dB(s) \rangle \notag \\
&\quad +\int_{0}^{\rho  } \psi'(\hat{\xi}(s))  \langle   (\nabla V)(s,X(s)), b(s,\cdot,X(s)) \rangle  \,ds\notag\\
&\quad +\int_{0}^{\rho  } \psi'(\hat{\xi}(s))  \langle   (\nabla V)(s, X(s)), d\Phi_{X}(s) \rangle\notag\\
&\quad +\frac12\int_{0}^{\rho   } \psi'(\hat{\xi}(s))\|\sigma(s,\cdot, X(s))\|^2 (\Delta V)(s,X(s))\,ds \notag  \\
&\quad  +\int_{0}^{\rho  } \psi'(\hat{\xi}(s)) \frac{\partial V}{\partial s}(s,X(s))\,ds
+\frac{1}{2} \int_{\R} \hat{L}_{\rho  }^{a}\psi''(da). \label{eq:eql2}
\end{align}
Here, $\{\hat{L}_s^a\}_{s \ge 0}$ denotes the local time at $a$ of the semimartingale $\{\hat{\xi}(s)\}_{s \ge 0}$, and $\psi''$ the second derivative of $\psi$ in the sense of distribution. Because $\psi$ is a concave function, it follows that 
\begin{align*}
 \int_{\R} \hat{L}_{\rho }^{a} \psi''(da) \le 0.
\end{align*}
From (V.3), the sum of the third, the fifth, and the sixth terms of \eqref{eq:eql2} is dominated by
\[
\int_{0}^{\rho  } \frac{g(s,\omega) \gamma(\hat \xi(s)) }{2\gamma(\hat{\xi}(s))}\, ds\le \frac{M}2. 
\]
Since $d\Phi_{X}(s)=\n(s)\,d|\Phi_{X}|_s$, where $\n(s) \in \cN_{X(s)}$ if $X(s) \in \partial D$, it follows from (V.2) that
the fourth term of \eqref{eq:eql2} is equal to
\[
\int_{0}^{\rho  } \psi'(\hat{\xi}(s))  \langle   (\nabla V)(s,X(s)), \n (s)\rangle\,d|\Phi_X|(s) \le 0.
\]
By using these estimates, we obtain
\begin{align}
\psi(\hat{\xi}(\rho)) &\le \psi(V(0,X(0)))+\int_{0}^{\rho }\psi'(\hat{\xi}(s))\langle  \sigma^{\ast}(s,\cdot,X(s))(\nabla V)(s,X(s)),\,dB(s)  \rangle\notag\\
&\quad+M/2.\label{eq:eql6}
\end{align} 
We define a local martingale $S=\{S(t)\}_{t \ge 0}$ as 
\[
S(t)=\int_{0}^{t \wg \tau_{R} \wg \chi_M}\psi'(\hat{\xi}(s))\langle  \sigma^{\ast}(s,\cdot,X(s))(\nabla V)(s,X(s)),\,dB(s)  \rangle,\quad t \ge 0.
\]
There exists an increasing sequence of stopping times $\{\theta_n\}_{n=1}^\infty$ such that $\lim_{n \to \infty}\theta_n= \infty$ $P$-a.s.\ and $\{S(t \wg \theta_n)\}_{t \ge 0}$ is a martingale for each $n \in \N$. Notice that \eqref{eq:eql6} is valid if we replace $\rho$ by $\rho \wg \theta_n$, $n \in \N$. Then, by Fatou's lemma, for any $t \ge 0$, $r>0$, $R>0$, and $M >0$,
\begin{align*}
&E[ \psi(\hat{\xi}(t \wg \tau_{R} \wg \chi_M)) :|X(0)|<r] \\
&\le \varliminf_{n \to \infty}E[ \psi(\hat{\xi}(t \wg \tau_{R} \wg \chi_M \wg \theta_n)) :|X(0)|<r] \\
&\le E[\psi(V(0,X(0))):|X(0)|<r] + \varliminf_{n \to \infty} E \left[ S(t \wg \theta_n):|X(0)|<r\right] +M/2 \\
&\le \sup_{x \in B(r)}V(0,x)+0+M/2.
\end{align*} 
Therefore, we have
\begin{equation}
E\biggl[\int_{0}^{\hat{\xi}(\tau_{R})}\frac{1}{\gamma(s)}\,ds : \zt_X \le t \wg \chi_M,\ |X(0)|<r\biggr]<\infty. \label{eq:eql8}
\end{equation} 
From (V.1),
\[
\lim_{R\to\infty}\hat\xi(\tau_R)=\infty \quad P\text{-a.s.\ on }\{\zt_X<\infty\}.
\]
Therefore, \eqref{eq:eql8} implies  that \[P(\zt_X \le t \wg \chi_M,\ |X(0)|<r)=0\] for any $t \ge 0$, $M>0$, and $r>0$.  
In view of \eqref{eq:g} and the fact that $X(0) \in \cl{D}$ $P$-a.s., we arrive at the conclusion. \qed
\end{proof}

\section{Proof of Theorem~\ref{thm:2}}
For a continuous functions $w\colon \R^d \to \R$, $s,t \in \R_{+}$ with $s<t$, and for $\theta \in (0,1]$, we define
\begin{align*}
\Delta_{s,t}(w)&= \sup_{s \le t_1<t_2 \le t}|w(t_2)-w(t_1)|,\\
\|w\|_{\mathcal{H},[s,t], \theta}&=\sup_{s \le t_1<t_2 \le t}\frac{|w(t_2)-w(t_1)|}{|t_2-t_1|^\theta}, \\
|w|_{t}^s&=\sup_{\Pi}\sum_{k=1}^{N}|w(t_k)-w(t_{k-1})|,
\end{align*}
where $\Pi=\{s=t_0<t_1<\cdots<t_{N}=t\}$ is a partition of the interval $[s,t]$.

Let $\{X(t)\}_{t\ge0}$ be a solution of \eqref{eq:sk3} with a Brownian motion $\{B(t)\}_{t\ge0}$. Define $\{W(t)\}_{0\le t<\zt_X}$ as
\[
W(t)=X(0)+ \int_{0}^{t}\sg(s,\cdot,X(s))\,dB(s)+\int_{0}^{t}b(s,\cdot,X(s))\,ds,\quad t <\zt_X.
\]
The following lemma is a slight modification of \cite[Lemma~2.3]{AS}, originally due to \cite[Proposition~3.1]{S}. The proof is the same as that of \cite[Lemma~2.3]{AS}.
\begin{lemma}\label{lem:AS}
Let $T>0$ and $\theta\in(0,1]$. There exist positive constants $C_1,C_2$ depending only on $\theta$, and $r_0$, $\beta$, $\dl$ in assumptions {\bf(A)} and {\bf(B)} such that, for $P$-a.s.,
\begin{align*}
|X|_{s}^{t}
&\le  C_1 (1+(t-s)\|W\|_{\mathcal{H},[s,t],\theta}^{1/\theta})e^{C_2\Delta_{s,t}(W)}\Delta_{s,t}(W),
\quad  0\le s<t< T\wg\zt_X.
\end{align*}
\end{lemma}
\begin{remark}
Following the proof of \cite[Lemma~2.3]{AS}, we can take 
\begin{align*}
C_1 &= 24\beta (1+\beta) \bigl[ \{ 4\dl^{-1}(\beta+2)\}^{1/\theta}+1 \bigr]\exp\{\beta \dl (1+\dl^{-1})r_0^{-1}\}\\
\shortintertext{and}
C_2 &= (1+\dl^{-1})\beta r_0^{-1}.
\end{align*}
However, we do not use such specific quantities below.
\end{remark}
In what follows, we suppose that the assumptions in Theorem~\ref{thm:2} are satisfied.
We introduce some stopping times and random integers. 
We fix $T,M \in (0,\infty)$ and set
 \begin{align*}
 \kappa=\inf\left\{t>0 \relmiddle| \int_{0}^t g(s,\cdot)\,ds \ge M \right\} \wg T.
 \end{align*}
For $R \in (0,\infty)$, we set
\begin{align*}
\kappa_R=\inf\{t>0 \mid |X(t)| \ge R \} \wg \kappa
\end{align*} 
and define 
 \[
 \kappa_\infty=\lim_{R\to\infty}\kappa_R\,(\,=\zt_X\wg\kappa).
 \]
We further define subsets $\{U_n\}_{n=0}^\infty$ and $\{V_n\}_{n=0}^\infty$ of $\cl{D}$ as
\begin{gather*}
U_n=B(x_n, \hat{\beta}\dl_n),\quad V_n=B(x_n,\dl_n), \quad n \ge 1, \\
\shortintertext{and}
U_0=\cl{D}\setminus \bigcup_{n=1}^\infty \cl{B}(x_n,\hat\beta \dl_n/2),\quad
V_0=\cl{D}\setminus \bigcup_{n=1}^\infty \cl{B}(x_n,\hat\beta \dl_n/3).
\end{gather*}
Note that $\cl{D}\subset \bigcup_{n=0}^\infty U_n$.
For $R \in \N \cup \{ \infty\}$, we define stopping times $\{\tau_k^{(R)}\}_{k=0}^\infty$ and random sequences $\{n_k^{(R)}\}_{k=0}^\infty$ as 
\begin{align*}
\tau_{0}^{(R)}&=0,\quad n_0^{(R)}=\inf\{ n \ge 0 \mid X(\tau_{0}^{(R)}) \in U_n \},
\end{align*}
and for $k\ge0$,
\begin{align*}
\tau_{k+1}^{(R)}&=\inf\{t> \tau_{k}^{(R)} \mid X(t) \notin V_{n^{(R)}_k} \} \wg \kappa_R,  \\
n_{k+1}^{(R)}&=\begin{cases}
 \inf\{ n \ge 0 \mid X(\tau_{k+1}^{(R)}) \in U_n \}&  \text{if}\quad \tau_{k+1}^{(R)}<\kappa_R,\\
\infty &  \text{if} \quad \tau_{k+1}^{(R)}=\kappa_R.
\end{cases}
\end{align*}
Let 
\begin{align*}
\Gamma_0&=\{k\ge0\mid n_k^{(\infty)}=0\}=\{k\ge0\mid n_k^{(\infty)}=0\text{ and } \tau_{k}^{(\infty)}<\kappa_{\infty}\},\\
\Gamma_1&=\{k\ge0\mid n_k^{(\infty)} \in \N\}=\{k\ge0\mid n_k^{(\infty)} \in \N \text{ and } \tau_{k}^{(\infty)}<\kappa_{\infty}\},\\
\shortintertext{and}
\Sigma&=\{k\ge0\mid n_k^{(\infty)} \in \N \text{ and } \tau_{k+1}^{(\infty)}<\kappa_{\infty}\}.
\end{align*}
\begin{lemma}\label{lem:cons1}
If $\#\Sigma$ is finite $P$-a.s., then both $\#\Gamma_0$ and $\#\Gamma_1$ are finite $P$-a.s.
\end{lemma}
\begin{proof}

Because $\tau_{k}^{(\infty)}$ is nondecreasing in $k$, there is at most one $k$ (depending on $\om\in\Omega$) such that $\tau_{k}^{(\infty)}<\kappa_{\infty}$ and $\tau_{k+1}^{(\infty)}=\kappa_{\infty}$. Therefore, $\#\Gamma_1\le\#\Sigma+1$ $P$-a.s.
Moreover, if $n_k^{(\infty)}=0$, then $n_{k+1}^{(\infty)}\in\N\cup\{\infty\}$ from the definition of $n_k^{(\infty)}$. Because there is at most one $k$ such that $n_k^{(\infty)}=0$ and $n_{k+1}^{(\infty)}=\infty$, we have $\#\Gamma_0\le \#\Gamma_1+1$ $P$-a.s. 
This completes the proof.
\qed
\end{proof}
\begin{lemma}\label{lem:cons3}
$\#\Sigma$ is finite $P$-a.s.
\end{lemma}
\begin{proof}
We define a sequence of random numbers $\{\tilde l_j\}_{j=1}^\infty\subset \N\cup\{0,\infty\}$ as
\begin{align*}
\tilde l_1&=\inf\{k\ge0\mid n_k^{(\infty)}\in\N\}\wg T,\\
\tilde l_{j+1}&=\inf\{k> l_j \mid n_k^{(\infty)}\in\N\}\wg T,\quad j\in\N.
\end{align*}
Also, let
\[
  l_j=\begin{cases}
  \tilde l_j & \text{if }\tau_{\tilde l_j+1}^{(\infty)}<\kappa_\infty\\
  T& \text{if }\tau_{\tilde l_j+1}^{(\infty)}=\kappa_\infty,
  \end{cases}
  \quad j\in\N.
\]
Hereafter, we omit the superscript $(\infty)$ and write $n_k$ and $\tau_k$ for $n_k^{(\infty)}$ and $\tau_k^{(\infty)}$, respectively.
For $n\in\N$, we write $M_n=M(x_n,\dl_n,T)$.
We set
\begin{align*}
B_j&=\{l_j<T\},\quad B_{j,n}=B_j\cap\{n_{l_j}=n\},\\
\tilde B_j&=\{\tilde l_j<T\},\quad \tilde B_{j,n}=\tilde B_j\cap\{n_{\tilde l_j}=n\},
\quad j\in\N,\ n\in\N.
\end{align*}
Note that $\{B_j\}_{j=1}^\infty$ is a decreasing sequence and $\bigcap_{j=1}^\infty B_j=\{\#\Sigma=\infty\}$.
On each $B_j$, $X(\tau_{l_j})\in \cl{B}(x_{n_{l_j}}, \hat \beta \dl_{n_{l_j}})$ and $X(\tau_{l_j+1})\notin B(x_{n_{l_j}}, \dl_{n_{l_j}})$, implying that 
\begin{align}\label{eq:distant}
|X(\tau_{l_j+1})-X(\tau_{l_j})|\ge (1-\hat \beta)\dl_{n_{l_j}}.
\end{align}
We fix $j\in\N$ and write $\tau=\tau_{l_j}$, $\hat \tau=(\tau_{l_j}+1/j)\wg \tau_{l_j+1}$, $x=x_{n_{l_j}}$, and $\dl=\dl_{n_{l_j}}$.
Both $\tau$ and $\hat\tau$ are $\{\mathcal{F}_t\}_{t\ge0}$-stopping times.
In the following, $c$ denotes an unimportant positive constant that may vary line-by-line.
Take $p$ such that $p>(1-\nu)^{-1}$.
Fix $n\in\N$. For $s,t$ with $0\le s<t$, we have
\begin{align}
&E[|W((\tau+t)\wg \hat\tau)-W((\tau+s)\wg \hat\tau)|^{2p}:B_{j,n}]\nonumber\\
&\le E[|W((\tau+t)\wg \hat\tau)-W((\tau+s)\wg \hat\tau)|^{2p}:\tilde B_{j,n}]\nonumber\\
&= E[|W((\tilde\tau+t)\wg \hat\tau)-W((\tilde\tau+s)\wg \hat\tau)|^{2p}],\label{eq:W-Bjn}
\end{align}
where
\[
\tilde\tau=\begin{cases}
\tau&\text{on }\tilde B_{j,n},\\
T&\text{on }\Omega \setminus \tilde B_{j,n}.
\end{cases}
\]
Since $\tilde B_j\in\mathcal{F}_\tau$ and
\[
\{n_{l_j}=n\}=\{\tau<\kappa_\infty\}\cap \left\{X_\tau\in U_n\setminus \bigcup_{k=0}^{n-1}U_k\right\}\in\mathcal{F}_\tau,
\]
$\tilde B_{j,n}=\tilde B_j\cap\{n_{l_j}=n\}\in\mathcal{F}_\tau$ and $\tilde\tau$ is an $\{\mathcal{F}_t\}_{t\ge0}$-stopping time.
Therefore, from the Burkholder--Davis--Gundy inequality, the last term of \eqref{eq:W-Bjn} is dominated by
\begin{align*}
&c E\Biggl[\left(\int_{(\tilde\tau+s)\wg\hat\tau}^{(\tilde\tau+t)\wg\hat\tau}\|\sg(s,\cdot,X(s))\|^2\,ds\right)^p+\left(\int_{(\tilde\tau+s)\wg\hat\tau}^{(\tilde\tau+t)\wg\hat\tau}|b(s,\cdot,X(s))|\,ds\right)^{2p}\Biggr]\\
&\le c\{(t-s) M_n\}^p P(\tilde B_{j,n})+c\{(t-s) M_n^{1/2}\}^{2p} P(\tilde B_{j,n})\\
&\le c M_n^p\{(t-s)^p+(t-s)^{2p}\}P(\tilde B_{j,n}).
\end{align*}
For $\alpha \in (1,\infty)$ and $\lambda \in (1/\alpha,1)$, the Garsia--Rodemich--Rumsey inequality~\cite[Corollary~A.2]{FV} (see also \cite[Lemma~1.1]{GRR} for the original inequality) 
implies 
\begin{align}
&E \left[\sup_{0 \le s<t\le 1/j}\left( \frac{|W((\tau+t)\wg\hat\tau)-W((\tau+s)\wg\hat\tau)|}{(t-s)^{\lambda-(1/\alpha)}} \right)^\alpha : B_{j,n}\right]  \notag \\
&\le c \int_{0}^{1/j}\!\! \int_{0}^{1/j} \frac{E \left[|W((\tau+t)\wg\hat\tau)-W((\tau+s)\wg\hat\tau)|^\alpha: B_{j,n} \right]}{|t-s|^{\alpha \lambda+1}}\,ds\,dt. \label{eq:lt-2}
\end{align}
Then, letting $\alpha=2p$ and $\lambda=1/\alpha+\nu/2\,(<1)$ in \eqref{eq:lt-2}, we obtain that
\begin{align}
E\left[\|W\|_{\mathcal{H},[\tau,\hat\tau],\nu/2}^{2p}: B_{j,n}\right]
&\le c M_n^{p}\int_{0}^{1/j}\!\! \int_{0}^{1/j} \frac{|t-s|^{2p}+|t-s|^p}{|t-s|^{2+p\nu}}\,ds\,dt\,P(\tilde B_{j,n})\nonumber\\
&\le c M_n^{p}\int_{0}^{1/j}\!\! \int_{0}^{1/j} |t-s|^{-2+p(1-\nu)}\,ds\,dt\,P(\tilde B_{j,n})\nonumber\\
&\le c M_n^{p} j^{-p(1-\nu)}P(\tilde B_{j,n}). \label{eq:W_Bjn}
\end{align}
Now, let $Z_j=B_j\cap\{\|W\|_{\mathcal{H},[\tau,\hat\tau],\nu/2}\ge M_{n_{l_j}}^{1/2}\}$. Then $Z_j=\bigsqcup_{n=1}^\infty Z_{j,n}$, where 
\[
Z_{j,n}=B_{j,n}\cap\{\|W\|_{\mathcal{H},[\tau,\hat\tau],\nu/2}\ge M_n^{1/2}\}.
\]
The Markov inequality and \eqref{eq:W_Bjn} lead us to
\[
P(Z_{j,n})
\le M_n^{-p}E[\|W\|_{\mathcal{H},[\tau,\hat\tau],\nu/2}^{2p}:B_{j,n}]
\le cj^{-p(1-\nu)}P(\tilde B_{j,n}).
\]
Thus, we have
\[
P(Z_j)=\sum_{n=1}^\infty P(Z_{j,n})\le cj^{-p(1-\nu)}P(\tilde B_j)\le cj^{-p(1-\nu)}.
\]
Since $p(1-\nu)>1$, we have $\sum_{j=1}^\infty P(Z_j)<\infty$.
From Borel--Cantelli's lemma, for $P$-a.s.\,$\om$ there exists $j_0=j_0(\om)$ such that for all $j\ge j_0$,
\begin{align}\label{eq:W-Bj}
\|W\|_{\mathcal{H},[\tau,\hat\tau],\nu/2}< M_{n_{l_j}}^{1/2} \quad\text{if }\om\in B_j.
\end{align}
Fix such $\om$ in $\bigcap_{j=1}^\infty B_j=\{\#\Sigma=\infty\}$. Let $j\in\N$, let $N$ be the smallest integer such that $N \ge \dl_{n_{l_j}}$, and write \[
t_k=\left(\tau_{l_j}+\frac{k}{j N}\right)\wg \tau_{l_j+1},\quad k=0,1,\dots, N.
\]
For an integer $k$ with $0\le k<N$ and $s,t$ with $t_k\le s<t\le t_{k+1}$, we have from \eqref{eq:W-Bj}
\[
|W(t)-W(s)|\le M_{n_{l_j}}^{1/2}(t-s)^{\nu/2}.
\]
This, together with \eqref{eq:thm2asmp}, implies that
\begin{equation}
\Delta_{t_k,t_{k+1}}(W)\le M_{n_{l_j}}^{1/2}(jN)^{-\nu/2} 
\le M_{n_{l_j}}^{1/2}(j \dl_{n_{l_j}})^{-\nu/2} 
\le c j^{-\nu/2}. \label{eq:dlN1}
\end{equation}
Applying Lemma~\ref{lem:AS} and using \eqref{eq:dlN1}, \eqref{eq:W-Bj}, and \eqref{eq:thm2asmp}, we obtain that
\[
|X|_{t_{k+1}}^{t_k}
\le C_1\left(1+\frac1{j N}(M_{n_{l_j}}^{1/2})^{2/\nu}\right)e^{C_2 c j^{-\nu/2}}c j^{-\nu/2} 
\le c j^{-\nu/2}
\]
for any integer $k$ with $0\le k<N$. Thus, we arrive at
\begin{align*}
|X|_{(\tau_{l_j}+1/j)\wg \tau_{l_j+1}}^{\tau_{l_j}}& \le \sum_{k=0}^{N-1}|X|_{t_{k+1}}^{t_k}
\le N c j^{-\nu/2}\\
&\le c(1+\hat\dl^{-1})\dl_{n_{l_j}}j^{-\nu/2},
\end{align*}
where $\hat\dl$ is a constant in the assumption of Theorem~\ref{thm:2}.
For sufficiently large $j$ (say, greater than or equal to $j_1=j_1(\om)\ge j_0(\om)$), $c(1+\hat\dl^{-1})j^{-\nu/2}<1-\hat \beta$. 
In view of \eqref{eq:distant}, it holds for such $j$ that $\tau_{l_j}+1/j<\tau_{l_j+1}$.
Then, for $P$-a.s.\,$\om$ in $\bigcap_{j=1}^\infty B_j=\{\#\Sigma=\infty\}$,
\[
T\ge \sum_{j=j_1(\om)}^\infty (\tau_{l_j+1}-\tau_{l_j})\ge \sum_{j=j_1(\om)}^\infty \frac1j=\infty,
\]
which is absurd. 
Therefore, $P(\#\Sigma=\infty)=0$.
\qed
\end{proof}

\begin{lemma}\label{lem:cons2}
For $k\ge0$,
\[
\sup_{\tau_{k}^{(\infty)}\le t <\tau_{k+1}^{(\infty)}}|X(t)|<\infty \quad P\text{-a.s.\ on }\{n_k^{(\infty)}=0\}.
\]
\end{lemma}
\begin{proof}
For $k \ge 0$ and $R \in \N \cup \{\infty\}$, we set 
\begin{align*}
A_k^{(R)}=\{ \om \in \Omega \mid n_k^{(R)}(\omega)=0\},
\end{align*}
which is $\mathcal{F}_{\tau_k^{(R)}}$-measurable.
For the moment, we fix $k\ge0$ and $R\in\N$, and suppress the superscript $(R)$ from the notation. 
We define 
\[
\xi_k(t)=|X(t\wg \tau_{k+1})-X(t\wg \tau_k)|^2,\quad t\ge0.
\]
For $t,u  \ge0$, we define 
\begin{align*}
\varphi(t,u)=\int_{0}^{t}\frac{1}{\gamma(2s+2u)}\,ds.
\end{align*}
For each $u \ge0$, $\R_+ \ni t \mapsto \varphi(t,u) \in \R$ is a concave function since $\gamma$ is a non-decreasing function on $\R_{+}$, so it can extend to a concave function on $\R$. This extension is still denoted by $\varphi$. We denote by $\varphi'$ the derivative of $\varphi(t,u)$ in $t$. Applying the It\^o--Tanaka formula to $\varphi (\cdot,u) \circ\xi_k$, we obtain for $u \ge 0$ that
\begin{align}
\varphi(\xi_k(\tau_{k+1}),u)&=\int_{\tau_k}^{\tau_{k+1}} \varphi'(\xi_k(s),u)\,d\xi_k(s)+\frac{1}{2} \int_{\R} (L_{\tau_{k+1}}^{a}-L_{\tau_k}^a) \varphi''(da,u)\notag\\
&=:I_1+I_2.\label{eq:cons1}
\end{align}
Here, $\varphi''(da,u)$ denotes the second derivative of $t \mapsto \varphi(t,u)$ in the sense of distribution. Then, $I_2\le0$.
We also have
\begin{align}
I_1&=2 \int_{\tau_k}^{\tau_{k+1}} \varphi'(\xi_k(s),u) \langle \sg^{\ast}(s,\cdot,X(s))(X(s)-X(\tau_k)),dB(s) \rangle \notag  \\
&\quad + \int_{\tau_k}^{\tau_{k+1}} \varphi'(\xi_k(s),u) \|\sg(s,\cdot,X(s))\|^2\,ds \notag \\
&\quad +2 \int_{\tau_k}^{\tau_{k+1}}\varphi'(\xi_k(s),u) \langle X(s)-X(\tau_k), b(s,\cdot,X(s))\rangle \,ds \notag  \\
&\quad +2 \int_{\tau_k}^{\tau_{k+1}} \varphi'(\xi_k(s),u) \langle  X(s)-X(\tau_k),  d\Phi_{X}(s) \rangle. \label{eq:I1}
\end{align}
On event $A_k$, we have $d\Phi_X=0$ on $[\tau_k,\tau_{k+1}]$ since $X(t) \notin \partial D$ for any $t \in [\tau_k, \tau_{k+1}]$. 
The first term of \eqref{eq:I1} is expressed as 
\[
2 \int_{0}^{\tau_{k+1}}  \varphi'(\xi_k(s),u) \langle \sg^{\ast}(s,\cdot,X(s))(X(s \wg \tau_{k+1})-X(s \wg \tau_k)),dB(s) \rangle.
\]
By \eqref{eq:cond2} and the fact that $\tau_{k+1} \le \kappa_R$, we see that
\[
\left\{  \int_{0}^{t}  \varphi'(\xi_k(s),|X(\tau_k)|^2) \langle \sg^{\ast}(s,\cdot,X(s))(X(s \wg \tau_{k+1})-X(s \wg \tau_k)),dB(s) \rangle \right\}_{t \in [0,\infty]}
\]
is a martingale by confirming that its quadratic variation is integrable at $t=\infty$. Therefore, by letting $u=|X(\tau_k)|^2$ in \eqref{eq:cons1} and taking the expectation of both sides of \eqref{eq:cons1} on $A_k$, we obtain
\begin{align*}
&E [\varphi(\xi_k(\tau_{k+1}),|X(\tau_k)|^2) :A_k]\notag  \\
&\le E\left[ \int_{\tau_k}^{\tau_{k+1}} \varphi'(\xi_k(s),|X(\tau_k)|^2) \|\sg(s,\cdot,X(s))\|^2\,ds :A_k \right] \notag \\
&\quad +2 \left[ \int_{\tau_k}^{\tau_{k+1}}\varphi'(\xi_k(s),|X(\tau_k)|^2) \langle X(s)-X(\tau_k), b(s,\cdot,X(s))\rangle \,ds :A_k  \right]\notag\\
&=:J_1+J_2.
\end{align*}
Recall that $\gamma(s)\colon\R_{+} \to \R_{+}$ is non-decreasing. 
By Lemma~\ref{lem:vee}(2), we may assume $\gamma(s)\ge s$ for $s \ge 0$ without loss of generality.
Then, we see from \eqref{eq:cond2}  that
\begin{align*}
J_1&\le E\left[ \int_{\tau_k}^{\tau_{k+1}} \frac{g(s,\cdot) \gamma(|X(s)|^2)}{\gamma(2|X(s)-X(\tau_k)|^2+2|X(\tau_k)|^2)}\,ds :A_k \right] \notag   \\
&\le E\left[ \int_{\tau_k}^{\tau_{k+1}} \frac{g(s,\cdot) \gamma(|X(s)|^2)}{\gamma(|X(s)|^2)}\,ds :A_k \right]
\le M
\end{align*}
and
\begin{align*}
J_2
&\le E\left[ \int_{\tau_k}^{\tau_{k+1}} \frac{|X(s)-X(\tau_k)|^2+|b(s,\cdot,X(s))|^2}{\gamma(2|X(s)-X(\tau_k)|^2+2|X(\tau_k)|^2)} \,ds :A_k \right] \notag \\
&  \le E\left[ \int_{\tau_k}^{\tau_{k+1}} \frac{|X(s)-X(\tau_k)|^2}{\gamma(|X(s)-X(\tau_k)|^2)}+\frac{|b(s,\cdot,X(s))|^2}{\gamma(|X(s)|^2)} \,ds :A_k \right] \notag \\
&  \le T+M.
\end{align*}
Combining these estimates, we obtain
\begin{align}
&E \left[\varphi\left(|X(\tau_{k+1}^{(R)})-X(\tau_k^{(R)})|^2,|X(\tau_k^{(R)})|^2\right) :A_k^{(R)}\right] \le T+2M. \label{eq:cons5}
\end{align}
It clearly holds that $\lim_{R \to \infty}A_k^{(R)}=\bigcup_{R\in\N}A_k^{(R)}=A_{k}^{(\infty)}$. By letting $R \to \infty$ in \eqref{eq:cons5} and using \eqref{eq:gm}, we obtain the conclusion. \qed
\end{proof}

\begin{proof}[of Theorem~\protect\ref{thm:2}]
By Lemmas~\ref{lem:cons1} and \ref{lem:cons3}, for $P$-a.s.\,$\om$, the number of $k \ge 0$ such that $n_k^{(\infty)}(\om)<\infty$ is finite. 
Let $\hat k=\hat k(\om)$ be the largest integer $k$ such that $n_k^{(\infty)}(\om)<\infty$.
If $n_{\hat k}^{(\infty)}(\om)=0$, then Lemma~\ref{lem:cons2} implies that
\begin{align}\label{eq:ztX}
 \zt_X(\om) \ge \tau_{\hat{k}+1}^{(\infty)}(\om)=\kappa_{\infty}(\om)=\kappa(\om). 
 \end{align}
  If $n_{\hat k}^{(\infty)}(\om)\in\N$,  \eqref{eq:ztX} clearly holds. Thus, it holds that $ \zt_X \ge \kappa$ $P$-a.s. Because $T$ and $M$ are arbitrarily chosen, we complete the proof from \eqref{eq:g}. \qed
\end{proof}


%



\begin{thebibliography}{}
%
%
\bibitem{AS}
Aida, S., Sasaki, K.: Wong-Zakai approximation of solutions to reflecting stochastic
   differential equations on domains in Euclidean spaces.  Stochastic Process. Appl. {\bf123}, 3800--3827 (2013)
\bibitem{BB}
Bass, R.F., Burdzy, K.: On pathwise uniqueness for reflecting Brownian motion in
   $C^{1+\gamma}$ domains. Ann. Probab. {\bf 36}, 2311--2331 (2008)
   \bibitem{BY}
   Bo, L., Yao, R.: Strong comparison result for a class of reflected stochastic
   differential equations with non-Lipschitzian coefficients. Front. Math. China {\bf 2}, 73--85 (2007)
   \bibitem{FZ} Fang, S., Zhang, T.: A study of a class of stochastic differential equations with
   non-Lipschitzian coefficients. Probab. Theory Related Fields {\bf 132}, 356--390 (2005)
   \bibitem{FV} Fritz, P.K., Victoir, N.B.: Multidimensional stochastic processes as rough paths: Theory and applications. Cambridge Studies in Advanced Mathematics, vol. 120, Cambridge University Press, Cambridge (2010)
   \bibitem{GRR} Garsia, A.M., Rodemich, E., Rumsey, H. Jr.: A real variable lemma and the continuity of paths of some Gaussian
   processes. Indiana Univ. Math. J. {\bf 20}, 565--578 (1970/1971) 
   \bibitem{L} Lan, G.Q.: Pathwise uniqueness and non-explosion of SDEs with
   non-Lipschitzian coefficients {\rm (in Chinese)}. Acta Math. Sinica (Chin. Ser.) {\bf 52}, 731--736 (2009)
   \bibitem{LS} Lions, P.L., Sznitman, A.S.: Stochastic differential equations with reflecting boundary
   conditions. Comm. Pure Appl. Math. {\bf 37}, 511--537 (1984)
   \bibitem{RS} Rozkosz, A., S\l omi\'{n}ski, L.: On stability and existence of solutions of SDEs with reflection at
   the boundary. Stochastic Process. Appl. {\bf 68}, 285--302 (1997)
   \bibitem{S} Saisho, Y.: Stochastic differential equations for multidimensional domain with
   reflecting boundary. Probab. Theory Related Fields {\bf 74}, 455--477 (1987)
   \bibitem{T} Tanaka, H.: Stochastic differential equations with reflecting boundary
   condition in convex regions. Hiroshima Math. J. {\bf 9}, 163--177 (1979)
    \bibitem{YW2}
    Watanabe, S., Yamada, T.: On the uniqueness of solutions of stochastic differential
   equations. II. J. Math. Kyoto Univ. {\bf 11}, 553--563 (1971)
   \bibitem{YW1} Yamada, T., Watanabe, S.: On the uniqueness of solutions of stochastic differential
   equations. J. Math. Kyoto Univ. {\bf 11}, 155--167 (1971)
   \bibitem{Z} Zhang, T.S.: On the strong solutions of one-dimensional stochastic differential
   equations with reflecting boundary. Stochastic Process. Appl. {\bf 50}, 135--147 (1994)
\end{thebibliography}


\end{document}